\documentclass[a4paper,11pt]{article}
\usepackage{graphicx}
\usepackage{enumitem}
\usepackage{fullpage}
\usepackage{tikz}
\usepackage{amssymb}
\usepackage{amsmath,amsthm}
\usepackage{verbatim}
\usepackage{amsfonts}
\usepackage[T1]{fontenc}
\usepackage{multicol}
\usepackage{color}
\usepackage{siunitx}
\usepackage{multirow}
\usepackage{hyperref}
\usepackage{yhmath}
\usepackage[utf8]{inputenc}
\usepackage{setspace}
\usepackage{geometry}

\hypersetup{
    colorlinks=true,
    linkcolor=blue,
    filecolor=magenta,      
    urlcolor=cyan,
    citecolor=red,
    pdftitle={Overleaf Example},
    pdfpagemode=FullScreen,
}
\newtheorem{thm}{Theorem}[section]
\numberwithin{equation}{section}
\newtheorem{prop}{Proposition}[section]

\newtheorem{lem}[thm]{Lemma}
\newtheorem{conj}[thm]{Conjecture}

\newtheorem{cor}[thm]{Corollary}

\newcommand{\bdot}{\boldsymbol{\cdot}}

\newcommand{\no}{\noindent}
\title{Combinatorial invariants for certain classes of non-abelian groups}
\author{Naveen K. Godara\footnote{Department of Mathematics, Indian Institute of Science Education and Research Bhopal, India. Email: naveen(dot)iiserb(at)gmail(dot)com}, Renu Joshi \footnote{Institut Computational Mathematics, Technische Universität Braunschweig,
38106 Braunschweig, Germany. Email: rjoshi(dot)iiser(at)gmail(dot)com} and Eshita Mazumdar\footnote{Mathematical and Physical Sciences Division, School of Arts and Sciences, Ahmedabad University, India. Email: eshita(dot)mazumdar(at)ahduni(dot)edu(dot)in}  }
\date{}

\begin{document}
\maketitle
\begin{abstract}
\no This article focuses on the study of zero-sum invariants of finite non-abelian groups. We address two main problems: the first centers on the ordered Davenport constant and the second on Gao's constant. We establish a connection between the ordered Davenport constant and the small Davenport constant for a finite non-abelian group of even order, which in turn gives a relation with the Noether number. Additionally, we confirm a conjecture of Gao and Li for a non-abelian group of order $2p^{\alpha}$, where $p$ is a prime. Furthermore, we prove a conjecture that connects the ordered Davenport constant to the Loewy length for certain classes of finite $2$-groups.
\end{abstract}

\no \textbf{Keywords:}
Gao constant, Davenport constant, Product-one sequence, Loewy length.
 
\smallskip

\no 2020 Mathematics Subject Classification. Primary 11B75; Secondary 11P70.

\section{Introduction}
For given positive integers $m<n$, we denote the set $\{m,m+1,\ldots,n\}$ by $[m,n].$ 
Let $G$ be a finite group, written multiplicatively with the identity element $1$, and $({\mathcal{F}}(G), \bdot)$ denote the free abelian monoid generated by $G.$ By a sequence $S$ over $G$, we mean $S\in{\mathcal{F}}(G).$ Typically, for a sequence $S=g_1\bdot g_2\bdot \dotsc \bdot g_{\ell}$ over $G$, the non-negative integer $|S|=\ell$ denotes the length of the sequence. For two sequences $S = g_1 \bdot \dotsc \bdot g_{\ell_1}$ and $T = h_1 \bdot \dotsc \bdot h_{\ell_2}$, we write $S \bdot T := g_1 \bdot \dotsc \bdot g_{\ell_1} \bdot h_1 \bdot \dotsc \bdot h_{\ell_2}$ in $\mathcal{F}(G)$.
 For any $g \in G$, we define the multiplicity of the element $g$ in $S$ as ${\mathsf{v}}_{g}(S) := \lvert \{ i \in [1,\ell] ~:~ g_i = g \}|. $

A sequence $T$ over $G$ is said to be a {\it subsequence} of $S$, denoted by $T \mid S$, if ${\mathsf{v}}_{g}(T) \leq {\mathsf{v}}_{g}(S)$ for every $g \in G$. We say that two sequences $S_1$ and $S_2$ over $G$ are the same if ${\mathsf{v}}_{g}(S_1)={\mathsf{v}}_{g}(S_2)$ for every $g \in G$. If we denote a sequence $S$ as $g_1^{[n_1]}\bdot g_2^{[n_2]}\bdot \dotsc \bdot g_{k}^{[n_k]}$, this means $g_i$ repeats $n_i$ times in $S.$ In case $T \mid S$ in $\mathcal{F}(G)$, we use the notation $S \bdot T^{[-1]}$ to denote the sequence obtained from $S$ by removing the terms of $T$. For a non-empty set $A\subset G$ and a sequence $S$ over $G$, we denote by $S(A)$ the subsequence of $S$ consisting of the elements of $A$ that appear in $S.$

\noindent For a sequence $S=g_1\bdot g_2\bdot \dotsc \bdot g_{\ell}$ over $G$, we define 
\[
\pi(S) := \bigg\{ \prod_{i=1}^{\ell}g_{\sigma(i)} \in G ~:~ \sigma ~ \in \mathfrak{S}_{\ell} \bigg\},
\] 
where, $\mathfrak{S}_{\ell}$ denotes the symmetric group. 
\[  
\Pi(S) := \bigcup_{ |T| \ge 1,\; T \mid S}  \pi(T).
\]

\no For any $r \in [1,\ell]$, we define
\[  
\Pi_r(S) := \bigcup_{\substack{|T|\ge 1,\; T \mid S,\\ |T|=r}}  \pi(T).
\]

\noindent 
It is quite useful to have related notations for sequences in which the order of terms matters. 
 Let $(\mathcal{F}^{\ast}(G), \bdot)$ denote the free non-abelian monoid with basis $G$ (as a set), whose elements are called ordered sequences over $G$, where we use the same notation "$\bdot$" for convenience. An ordered sequence $T^{\ast}$ over $G$ is said to be an {\it ordered subsequence} of $S^{\ast}=g_1\bdot g_2 \bdot g_3 \bdot \dots \bdot g_{\ell}$ $\in \mathcal{F}^{\ast}(G)$, denoted by $T^{\ast} \mid S^{\ast}$, if $T^{\ast}=g_{i_1}\bdot g_{i_2} \bdot g_{i_3} \bdot \dotsc \bdot g_{i_k}$ for some $1\le i_1<i_2<i_3 < \dots <i_k\le \ell$. 
 For a non-empty set $A\subset G$ and a sequence $S^{\ast}$ over $G$, we denote by $S^{\ast}(A)$ the subsequence of $S^{\ast}$ consisting of the elements of $A$ following the ordering of terms as they appear in $S^{\ast}.$
 For an ordered sequence $S^{\ast}=g_1\bdot g_2 \bdot g_3 \bdot \dots \bdot g_{\ell}$, we define $\pi(S^{\ast}) := \prod_{i=1}^{\ell}g_{i} \in G,
$
and 
\begin{equation*}
\Pi(S^{\ast}) :=\big\{ \pi(T^{\ast}) ~:~  |T^{\ast}|\ge 1,\; T^{\ast} \mid S^{\ast} \big\} \subseteq G.
\end{equation*}
Note that if $S \in \mathcal{F}(G),$ then corresponding ordered sequence of $S$ is $S^{\ast} \in \mathcal{F}^{\ast}(G).$
Now, we define some combinatorial invariants. Let $S$ be a non-trivial sequence over $G$.
\begin{itemize}
    \item The sequence $S$ is said to be a {\it product-one sequence} if $1 \in \pi(S)$.
    \item The sequence $S$ is called a {\it product-one free sequence} if  $1\not\in \Pi(S)$.
    \item The {\it small Davenport constant} denoted by $\mathsf{d}(G)$ is defined as $$\mathsf{d}(G):= \sup\big\{|S| : \exists\; S\in {\mathcal{F}}(G), 1\notin \Pi(S) \big\}.$$
    \item The {\it Gao constant} denoted by $\mathsf{E}(G)$ is defined as 
    $$\mathsf{E}(G):= \min\big\{k\in \mathbb{N} : |S| \ge k \text{ implies } 1\in \Pi_{|G|}(S) \big\}.$$
\end{itemize}
On the other hand, let $S^{\ast}$ be a non-trivial ordered sequence over $G$. 
 \begin{itemize}
     \item The sequence $S^{\ast}$ is said to be a {\it product-one ordered sequence} if $\pi(S^{\ast})=1$. We say $S^{\ast}$ has a product-one ordered subsequence if $1 \in \Pi(S^{\ast})$. 
     \item The sequence $S^{\ast}$ is called a {\it product-one free ordered sequence} if $1\not\in \Pi(S^{\ast})$.
\item The {\it ordered Davenport constant} denoted by $\mathsf{D}_o(G)$ is 
defined as $$\mathsf{D}_o(G)=\min\big\{k \in \mathbb{N} : |S^{\ast}|\ge k \text{ implies } 1\in \Pi(S^{\ast}) \big\}.$$
\end{itemize}

\no Note that in \cite{Dimitrov}, $\mathsf{D}_o(G)$ is referred to as the strong Davenport constant, but in the literature, the strong Davenport constant \cite{chapman} is meant for something else. So, to avoid confusion, we call it the ordered Davenport constant. In the past, $\mathsf{d}(G)$ has been studied extensively, but not much is known about $\mathsf{D}_o(G)$. Determining the precise value of $\mathsf{D}_o(G)$ is a considerably more challenging and difficult problem. Recently, in \cite{NRE}, the authors computed the precise value of $\mathsf{D}_o(G)$ for certain classes of finite non-abelian $p$-groups and several other groups as well.

\no By the definitions of $\mathsf{d}(G)$ and $\mathsf{D}_o(G)$, it is clear that both the invariants are natural extensions of the Davenport constant \cite{Davenport}, and  $\mathsf{d}(G)+1 \le \mathsf{D}_o(G) \le |G|$ for any finite group $G$ (equality holds for any finite cyclic group $G$). For a finite abelian group $G$, both $\mathsf{d}(G)+1$ and $\mathsf{D}_o(G)$ are the same. In 1977, Olson and White \cite{OW} proved that $\mathsf{D}_{o}(G) \le \bigg\lceil \frac{|G|+1}{2} \bigg\rceil$ for any finite non-cyclic group $G$. Since $\mathsf{d}(G) +1 \le \mathsf{D}_o(G)$ for any finite group $G$, we have $\mathsf{d}(G) \le \big\lfloor \frac{|G|}{2}\big\rfloor$. So, to study the ordered Davenport constant, we focus on the relationship between $\mathsf{d}(G)$ and $\mathsf{D}_o(G)$ in detail.

\no Moreover, in the realm of invariant theory, an important constant to consider is the Noether number, which is related to zero-sum invariants. For a finite-dimensional $G$-module $V$ over a field $\mathbb{F}$, where $\text{char}(\mathbb{F}) \nmid |G|$, the ring of invariants is defined as $\mathbb{F}[V]^G:=\{f \in \mathbb{F}[V] : f^g=f ~\forall ~g \in G\}$.
Let $\beta(G, V)$ be the smallest positive integer $d$ such that $\mathbb{F}[V]^G$ is generated by elements of degree $\le d$. The \textit{Noether number} $\beta(G)$ is then defined as: 
$$\beta(G) = \sup{\{\beta(G,V) : V \text{ is a finite-dimensional } G \text{-module over } \mathbb{F}\}}.$$
\no Schmid \cite{Schmid} proved a connection between the small Davenport constant and the Noether number; in particular, $\mathsf{d}(G)+1 = \beta(G)$ for a finite abelian group $G$ (see \cite{CD2014} for more details).

\smallskip

\no In this connection, we have our first result, which is as follows:
\begin{thm}\label{Theoremd(G)}
  Let $A$ be a finite abelian group. For the group $G= A \rtimes_{-1} C_2$, we have 
    $$\mathsf{D}_o(G) = \mathsf{d}(G)+1 = \mathsf{d}(A)+2.$$
    Moreover, $\mathsf{D}_o(G)=\beta (G).$
\end{thm}

\no Our next result is related to the Gao constant. In 1961, Erd\H{o}s, Ginzburg, and Ziv \cite{EGZ1961} proved that ${\mathsf{E}}(G) \leq 2|G| - 1$ for any finite solvable group $G$, which is known as the Erd\H{o}s-Ginzburg-Ziv theorem. Later, in 1976, Olson \cite{OLS1976} extended this result to any finite group $G$. Following their work, significant efforts were made to reduce the upper bound on ${\mathsf{E}}(G)$ (see \cite{YUS1988} and \cite{GAO1996B} for more details). In 1996, Gao established that $\mathsf{E}(G)=\mathsf{d}(G)+|G|$ for any finite abelian group $G$. Later on, Gao and Li \cite{GLI2010}, in 2010, refined the upper bound to ${\mathsf{E}}(G) \leq {\frac {7}{4}} |G| - 1$ and proposed a conjecture:

\bigskip

\begin{conj}\emph{\cite{GLI2010}}\label{GAOLI} For any finite non-cyclic group $G$, we have ${\mathsf{E}}(G) \leq {\frac {3}{2}}|G|$.
\end{conj}
\noindent In 2015, Han \cite{HanD} confirmed that Conjecture \ref{GAOLI} holds for all finite non-cyclic nilpotent groups (see \cite{NKG}, \cite{ABR2023}, \cite{DHA2019}, and \cite{YYU2023} for recent progress on the above conjectures). Recently, in 2021, Gao, Li, and Qu \cite{GaoLiQu} modified the bound in Conjecture \ref{GAOLI} and proved that ${\mathsf{E}}(G) \leq {\frac {3}{2}}(|G| - 1)$ for a finite non-cyclic group $G$ of odd order $|G| > 9$.

\no In support of Conjecture \ref{GAOLI}, we prove the following result for $\mathsf{E}(G)$:

\begin{thm}\label{TheoremE(G)} Let $p$ be a prime and $A$ be a finite abelian $p$-group. For the group $G= A \rtimes_{-1} C_2$, we have 
    $$\mathsf{E}(G) \le \frac{3}{2}|G|.$$
\end{thm}

\no In the final result of this article, we establish a connection between the ordered Davenport constant and the Loewy length. Let $p$ be a prime number. 

The nilpotency index of the Jacobson radical $J$ of $\mathbb{F}_p[G]$ is known as the {\it Loewy length} of $\mathbb{F}_p[G]$ and we denote it by $\mathsf{L}(G)$. For any finite $p$-group $G$, the Loewy length $\mathsf{L}(G)$ is bounded above by $|G|$. In 2004, Dimitrov explored $\mathsf{D}_o(G)$ for a finite $p$-group $G$ and proposed a conjecture relating $\mathsf{D}_o(G)$ and $\mathsf{L}(G)$:
\begin{conj}\emph{\cite{Dimitrov}}\label{Dimitrovconj}
    For a prime $p$ and a finite $p$-group $G$, we have $\mathsf{D}_o(G)= \mathsf{L}(G).$
\end{conj}

\no This conjecture is true for finite abelian $p$-groups. Recently, in \cite{NRE}, authors proved that this conjecture holds for a large subclass of finite non-abelian $p$-groups when $p$ is an odd prime. Consequently, authors also improved the upper bound on the small Davenport constant given in \cite[Theorem 1.1]{QuLiTeeuwsen2022}. In this paper, we extend our findings to show that Conjecture \ref{Dimitrovconj} holds for a subclass of finite non-abelian $2$-groups. Consider the following $2$-groups (appeared in \cite{BK}):
\begin{enumerate}[label={(\roman*)}]
    \item  $G_1 =  (\langle c \rangle\times \langle a\rangle)\rtimes \langle b\rangle, \text{ where }[a,b]=c,\; [a,c]=[b,c]=1,\;$
    $o(a)=2^{\alpha},\; o(b)=2^{\beta},\\ o(c)=2^{\gamma}, \;
    \alpha, \beta, \gamma \in \mathbb{N} \;\text{with}\; \alpha \geq \beta \geq\gamma \geq 1.$ 

\item $G_2 = \langle a \rangle\rtimes \langle b \rangle, \text{ where } [a,b]=a^{2^{\alpha-\gamma}},\;
    o(a)=2^{\alpha}, \; o(b)=2^{\beta}, o([a,b])=2^{\gamma},$
    $\alpha, \beta,\gamma \in \mathbb{N} \newline \text{ with } \alpha \geq 2\gamma,\; \beta  \geq \gamma \geq 1, \; \alpha+\beta>3.$
    
\item $G_3 =  (\langle c \rangle\times \langle a\rangle)\rtimes \langle b\rangle, \text{ where } [a,b]=a^{2^{\alpha-\gamma}} c,\; 
    [c,b]= a^{-2^{2(\alpha-\gamma)}}c^{-2^{\alpha-\gamma}},$ $o(a)=2^{\alpha},\newline o(b)=2^{\beta},  o(c)=2^{\sigma},\; o([a,b])=2^{\gamma}, \;
\alpha, \beta, \gamma, \sigma \in \mathbb{N} \;\text{with}\; \beta \ge \gamma > \sigma \geq 1,\; \alpha +\sigma \geq 2\gamma.$ 

\item $G_4 =  (\langle c \rangle\times \langle a\rangle)\langle b\rangle, \;\text{where}\; o(a)=o(b)=2^{\gamma+1},\; o([a,b])=2^{\gamma},\; o(c)=2^{\gamma-1},\;[a,b]=a^{2} c,\newline\; [c,b]= a^{-4}c^{-2},\; a^{2^{\gamma}}=b^{2^{\gamma}},\; \gamma \in \mathbb{N}$.
\end{enumerate}
We want to point out that the group $G_4$ is not a semidirect product of the groups $(\langle c \rangle\times \langle a\rangle)$ and $\langle b\rangle$.
 Instead, it is a non-split extension, and the notation used here follows the convention in \cite{BK}.

\begin{thm}\label{TheoremOrderedLoewy} For the above groups, we have $\mathsf{D}_o(G)=\mathsf{L}(G)$ if either
\begin{enumerate}[label=\emph{(\arabic*)}]
\item  $G\cong G_1$ for $\gamma=1$.  
\item $G\cong G_2.$
\item $G\cong G_3$ for $\sigma=1$.
\item $G\cong G_4$ for $\gamma \in \{1,2\}$.
\end{enumerate}
\end{thm}

\begin{cor}\label{d}
    If $G$ is a finite group such that $G\cong G_2,$ then $\mathsf{D}_o(G)=\mathsf{d}(G)+1.$
\end{cor}

\no In the above connection, we also have the following result: 
\begin{thm}\label{LemmaIndexp}
    Let $G$ be a finite non-cyclic $p$-group. Then $G$ has a cyclic subgroup of index $p$ if and only if $$\mathsf{d}(G)+1=\mathsf{D}_o(G)=\mathsf{L}(G)=\frac{|G|}{p}+p-1.$$
     
\end{thm}

\no The rest of the paper is organised as follows: We begin with some preliminaries, which contain useful lemmas and propositions. In Section \ref{SectionTheoremproof1.11.3}, we prove our main results, Theorem \ref{Theoremd(G)} and Theorem \ref{TheoremE(G)}. In Section \ref{SectionL(G)}, we focus on evaluating precise value of $\mathsf{L}(G_i)$ for all $i\in [1,4].$ In the next section, we give a proof of Theorem \ref{TheoremOrderedLoewy} and Theorem \ref{LemmaIndexp}. We conclude the paper with an open question.

\section{Preliminaries}\label{section pre}

\noindent We fix some standard notations. The commutator of elements $x, y$ in a group $G$ is defined as $[x,y]:=x^{-1}y^{-1}xy$. If $H_i$ is a subgroup of $G$ for $i=1,2,$ then $[H_1,H_2]$ denotes the subgroup of $G$ generated by commutators $[h_1,h_2],$ where $h_i \in H_i$ for $i=1,2.$
For any $n\in \mathbb{N}$, $G^{n}:=\langle x^n \mid x\in G \rangle$ is a subgroup of $G.$ If $G$ is a finite $p$-group, then the subgroups $G^{p^n}$ form a descending chain of subgroups of $G$. The minimum number of generators of a group $G$ is denoted by $r(G)$. 

\no Let $p$ be a prime number, and let $G$ be a finite $p$-group. Since the modular group algebra $\mathbb{F}_p[G]$ has a unique two-sided maximal ideal, the Jacobson radical $J$ coincides with the augmentation ideal generated by $\{g-1 \mid g \in G \setminus \{1\}\}$ of $\mathbb{F}_p[G].$
 The dimensions of the Loewy factors $J^i/J^{i+1}$ have been computed by Jennings (see \cite{J} for more details) in terms of the Brauer-Jennings-Zassenhaus series (or $M$-series) $M_i(G)$ of $G$, where
\begin{eqnarray*}
    M_{1}(G) & := & G,\\
    M_{i}(G) & := & [M_{i-1}(G), G]M_{\lceil \frac{i}{p} \rceil}(G)^{p} \text{ for }i \geq 2.
\end{eqnarray*}
\noindent For convenience, we denote $M_i(G)$ by $M_i$ for all $i \ge 1,$ whenever the underlying group $G$ is understood. Through induction, it can be shown that $\{M_i\}$ is a decreasing sequence of characteristic subgroups of $G$. Next, given any $n \in \mathbb{N}$, there exists $r_n \in \mathbb{N}$ with $r_n > n$ such that $M_{r_n} \subsetneq M_n$. So, there exists a natural number $d$ such that $M_{d}\neq 1$, and $M_{d+1}=1.$ Note that while $M_i\neq 1$, the equality $M_i = M_{i+1}$ might hold for some indices $i$. We write:
$$|M_i/M_{i+1}|=p^{e_i} \; \; \text{ for some natural number } e_i,  \text{ for all } i\in [1,d].$$

\no In 1941, Jennings provided an explicit formula to compute $\mathsf{L}(G)$.

\begin{lem}\emph{\cite{J}}\label{JenningsTheorem} Let $G$ be a finite $p$-group, and let $\{M_i\}$ denote the $M$-series of $G$. Then\begin{enumerate}[label=\emph{(\roman*)}]
        \item For any index $i$, the quotient $M_i/M_{i+1}$ is a finite elementary abelian $p$-group.
        \item  The Loewy length $\mathsf{L}(G)=1+(p-1)\sum_{i=1}^d ie_i$, where $d$ is the largest positive integer with $M_{d}\neq 1$ and $p^{e_i}$ is the order of $M_i/M_{i+1}$.
    \end{enumerate}
\end{lem}

\begin{lem}\emph{{\cite[Theorem 1, Corollary 1]{Dimitrov}}}\label{Dimitovtheorem}
Let $G$ be a finite $p$-group. Then $\mathsf{D}_o(G) \le \mathsf{L}(G).$
\end{lem}

\noindent Note that the group $G=G_i$ for $i\in [1,4],$ satisfies
\begin{equation}
[G,G] \neq 1 \quad \text{and}\quad [[G,G],G]=1. \label{2.1}
\end{equation}
Therefore, $(G^{2^i})^{2^j}= G^{2^{i+j}}$ for any positive integers $i,j$. Furthermore, since  $[G,G]\subseteq G^2$ for any $2$-group $G$, it follows that  $[G,G]^{2^s}\subseteq G^{2^{s+1}}$ for all $s\geq 0.$ The following result is essential for determining the $M$-series.

\begin{prop}\label{Ind} 
Let $G$ be a $2$-group satisfying relation \emph{(\ref{2.1})}. Then, for all $s \geq 1$, we have 
\begin{equation*}\label{equationinduction}
M_j=G^{2^{s}} \;\text{ for all }\; j \in [2^{s-1}+1, 2^s].
\end{equation*}
\end{prop}
\begin{proof} Assume that
$P(s) :  M_{2^{s-1}+1}= \cdots =M_{2^{s}}=G^{2^{s}} \; \text{for} \; s \ge 1.$ We will use induction on $s \ge 1$. From the definition of the $M$-series of $G$, we have $M_2=[G,G]G^2$. Since $[G,G]\subseteq G^2$, it follows that $P(1)$ holds. Assume that $P(s)$ is true for some $s\ge 1$, i.e.
\begin{equation}\label{eq:indhyp}
M_{2^{s-1}+1}=\cdots =M_{2^{s}}=G^{2^{s}}.
\end{equation}
We show that $P(s+1): M_{2^{s}+1}=\cdots=M_{2^{s+1}}=G^{2^{s+1}}$ holds. First, using the recursive definition 
\[
M_{2^s+1}
=[M_{2^s},G]\, M_{\lceil (2^s+1)/2\rceil}^{2},
\]
 and since $\lceil (2^{s}+1)/2\rceil=2^{s-1}+1$, the induction hypothesis
\eqref{eq:indhyp} gives $M_{2^{s}}=G^{2^{s}}$ and $M_{2^{s-1}+1}=G^{2^{s}}$. Hence
$M_{2^{s}+1}
=[G^{2^{s}},G]\,(G^{2^{s}})^{2}
=G^{2^{s+1}},$ because 
$[G^{2^{s}},G]\subseteq G^{2^{s+1}}$. Now, assume that $$M_t=G^{2^{s+1}} \text{ for all } t \in [2^{s}+1, k]\quad \text{for some $k$ with $2^{s}+1\le k<2^{s+1}$}.$$ Then \[
M_{k+1}=[M_k,G]\; M_{\lceil (k+1)/2\rceil}^{2} = [G^{2^{s+1}},G]\; M_{\lceil (k+1)/2\rceil}^{2}.
\]
Since $k+1\le 2^{s+1}$, it follows that $\lceil (k+1)/2\rceil\in[2^{s-1}+1,2^{s}]$, and the induction hypothesis \eqref{eq:indhyp} implies $M_{\lceil (k+1)/2\rceil}=G^{2^{s}}$. Therefore
\[
M_{k+1} = [G^{2^{s+1}},G]\,(G^{2^{s}})^2 = G^{2^{s+1}}, \text{ since $[G^{2^{s+1}},G]=[G,G]^{2^{s+1}}\subseteq G^{2^{s+2}}\subseteq G^{2^{s+1}}.$}
\] By induction on $t$, we conclude that 
$M_{2^{s}+1}=\cdots=M_{2^{s+1}}=G^{2^{s+1}},$ and hence $P(s+1)$ holds.
\end{proof}

\noindent In addition to the above proposition, the following result plays an important role in computing the order of the quotient groups $M_i/M_{i+1}$ within $M$-series, which is crucial for determining the precise value of $\mathsf{L}(G)$.

\begin{prop}{\label{power}}
Let $G=\langle a,b \rangle$ be a $2$-group satisfying relation \emph{(\ref{2.1})}. Then $G^{2^s}=\langle a^{2^s}, b^{2^s}, [a,b]^{2^{s-1}} \rangle$ for all $s\geq 1$.
\end{prop}
\begin{proof} Clearly, $\langle a^{2^s}, b^{2^s}, [a,b]^{2^{s-1}} \rangle \subset G^{2^s} $ for all $s\geq 1$.
Now, let $y \in G^{2^s}$. Then $y=x^{2^s}_1x^{2^s}_2 \cdots x^{2^s}_k$ for some $k\in \mathbb{N}$, where each $x_{i}=a^{t_i} b^{u_i} [a,b]^{v_i}$ for some $t_i, u_i, v_i \in \mathbb{Z}$. Therefore, $x^{2^s}_{i}=(a^{t_i} b^{u_i} [a,b]^{v_i})^{2^s}=a^{2^s t_i}b^{2^s u_i}[a,b]^{2^{s-1}({2v_i}-{t_{i}u_{i}(2^{s}-1))}}$ for each $i$, the result follows.
\end{proof}

\begin{lem}\emph{\cite[Theorem 1.6]{Koshitani1977}}\label{EqL}
    Let $G$ be a finite $p$-group for prime $p.$ Then $\mathsf{L}(G) = \frac{|G|}{p}+p-1$ if and only if $G$ is non-cyclic and has a cyclic subgroup of index $p.$
\end{lem}

\begin{lem}\emph{\cite[Theorem 1.1]{QuLiTeeuwsen2022}}\label{dd}
    Let $G$ be a finite non-cyclic group and $p$ the smallest prime divisor of $|G|.$
Then $$\mathsf{d}(G) \le  \frac{|G|}{p}+p-2 $$with equality if $G$ contains a cyclic subgroup of index $p.$
\end{lem}
\no The following results are useful in obtaining the upper bound on the Gao constant $\mathsf{E}(G)$ in Theorem \ref{TheoremE(G)}:

\begin{lem}\label{d(G)D_o(G)}
    Let $G$ be a finite group. Then 
    $$\mathsf{d}(G)+1 \le \mathsf{D}_o(G) \le |G|.$$
    Moreover, $\mathsf{d}(G)+1 = \mathsf{D}_o(G)$ when $G$ is abelian, and $\mathsf{D}_o(G)=|G|$ if and only if $G$ is cyclic.
    \end{lem} 
\begin{proof} 

It is enough to prove the second part of the statement. If $G=\langle g \rangle$, then consider the sequence $S=g^{[|G|-1]}$ to have $ \mathsf{D}_o(G) = |G|.$ On the other hand, let us assume $\mathsf{D}_o(G)=|G|$. Since $\mathsf{D}_{o}(G) \le \bigg\lceil \frac{|G|+1}{2} \bigg\rceil$ for a finite non-cyclic group $G$, we are done.
\end{proof}

\begin{lem}\emph{\cite[Corollary 5.5]{CD2014}}\label{beta}
    For any finite abelian group $A$, we have $\beta (A \rtimes_{-1} C_2)= \mathsf{d}(A)+2.$
\end{lem}

\begin{lem}\emph{\cite[Proposition 5.7.7]{GerolHalt}}\label{n,2n}
    Let $G$ be a finite abelian group and $n\in \mathbb{N}$. Let $S$ be a sequence over $G,$ and $\mathsf{d}(G \times C_n)\le 3n-2.$ If $|S|\ge \mathsf{d}(G\times C_n)+1$ then $S$ has a product-one subsequence $T$ over $G$ with $|T|\in \{n,2n\}.$
\end{lem}

\begin{lem}\emph{\cite[Theorem 1.4]{HanD}}\label{Han}
    Let $G$ be a finite nilpotent non-cyclic group of order $n$, and let
$p$ be the smallest prime divisor of $n$. Then
$$\mathsf{E}(G) \le n +\frac{n}{p} + p -2.$$ In particular, $\mathsf{E}(G)\le \frac{3n}{2}.$
\end{lem}

\begin{lem}\emph{\cite[Lemma 7]{GaoLu2008}}\label{pi_n-2}
    Let $G$ be a finite abelian group of order $n$, $r \ge 2$ an integer, and $S$ a sequence of $n + r - 2$ elements in $G$. If $1 \notin \Pi_{n}(S)$, then $|\Pi_{n-2}(S)|=|\Pi_{r}(S)| \ge r-1$.
\end{lem}

\begin{lem}\emph{\cite[Lemma 7]{BAS2007}}\label{disjointsequences}
    Let $S$ be a sequence over an abelian group $G$ of order $n$. For any integer $k$ such that
$2^k >n$, there exist subsequences $U,V \mid S$ such that\\
\no \emph{(i)} $U$ and $V$ are disjoint sequences.\\
     \emph{(ii)} $|S| - \big(|U| + |V|\big) \le k - 1$, and\\
     \emph{(iii)} $\pi(U)=\pi(V)$.

\end{lem}

\begin{lem}\emph{\cite[Lemma 2.2]{Nathanson1996}}\label{Nathanson}
    Let $A, B$ be two subsets of a finite abelian group $G$. If $|A| + |B| > |G|$, then $$A + B:=\{ab ~|~a \in A, b \in B\} = G.$$
\end{lem}

\section{Proof of Theorem \ref{Theoremd(G)} and Theorem \ref{TheoremE(G)}}\label{SectionTheoremproof1.11.3}

For a finite abelian group $A$ of rank $r,$ there exists $n_i\in \mathbb{N}$ for $i\in [1,r]$ such that
$$A \cong C_{n_1}\times C_{n_2}\times \dotsc \times C_{n_r}$$ with $n_i \mid n_{i+1}$ for $i\in [1,r-1].$ We consider the group $G=A \rtimes_{-1} C_2$, then
$$G \cong \langle g_1,g_2,\dotsc g_r,h \mid g_1^{n_1}=g_2^{n_2}=\dotsc =g_{r}^{n_r}=1, \; h^2=1, \; hg_ih=g_i^{n_i-1}, \; g_ig_j=g_jg_i \rangle.$$
Let us define the set $H:=G \setminus A$, i.e., $H=\{g_1^{i_1}g_2^{i_2}\dotsc g_r^{i_r}h \mid i_k \in [0, n_k-1] \; \text{for} \; k \in [1,r]\}$. Note that $c_1c_2 \in A$ for any $c_1, c_2 \in H$, and every element of the set $H$ has order $2$.

\subsection*{Proof of Theorem \ref{Theoremd(G)}}

Clearly, for a finite abelian group $A$, there exists a sequence $S$ over $A$ of length $\mathsf{d}(A)$ that is product-one free. Consider the sequence $T=S\bdot h$ over $G$, then $ |T|=\mathsf{d}(A)+1$. Then $T$ is also a product-one free sequence as $h$ can not be expressed as an element of $A.$

This indicates that $\mathsf{d}(A)+1 \le \mathsf{d}(G)$, and by Lemma \ref{d(G)D_o(G)}, we derive that $\mathsf{d}(A)+2 \le \mathsf{d}(G)+1 \le \mathsf{D}_o(G)$.

\no Now, it is enough to show that $\mathsf{D}_o(G) \le \mathsf{d}(A)+2.$ So, we consider a sequence $S^{\ast} \in \mathcal{F}^{\ast}(G)$ of length $\mathsf{d}(A)+2$. 

\no \textbf{Case (i) :} If all the terms of $S^{\ast}$ are from $A$. From the definition of $\mathsf{d}(A),$ it follows that $S^{\ast}$ contains a non-trivial product-one ordered subsequence.

\no \textbf{Case (ii) :} Assume that all terms of $S^{\ast}$ lie in $H$. Write 
\[
S^{\ast}
= c_1 \bdot c_2 \bdot \dotsb \bdot c_{\mathsf{d}(A)+1} \bdot c_{\mathsf{d}(A)+2},
\qquad c_i \in H.
\]
We associate to $S^{\ast}$ the ordered sequence $T^{\ast} \in \mathcal{F}^{\ast}(A)$ defined by 
\[
T^{\ast}
= (c_1c_2) \bdot (c_2c_3) \bdot \dotsb \bdot (c_{\mathsf{d}(A)}c_{\mathsf{d}(A)+1}) 
 \bdot (c_{\mathsf{d}(A)+1}c_{\mathsf{d}(A)+2}),
\]
so that $|T^{\ast}|=\mathsf{d}(A)+1$. Since $\mathsf{D}_o(A)=\mathsf{d}(A)+1$ by Lemma~\ref{d(G)D_o(G)}, the sequence $T^{\ast}$ contains a non-trivial product-one ordered subsequence. Denote such a subsequence by
\[
(T')^{\ast}
= (c_{j_1}c_{j_2}) \bdot (c_{j_3}c_{j_4}) \bdot \dotsb \bdot (c_{j_k}c_{j_{k+1}}).
\]

\noindent
Since every $c_{j_i} \in H$ has order $2$, any adjacent overlapping pairs of the form
\[
(c_{j_i}c_{j_\ell}) \bdot (c_{j_\ell}c_{j_m}) \;\mid\; (T')^{\ast}
\]
may be replaced by the ordered subsequence
\[
c_{j_i} \bdot c_{j_m} \;\mid\; S^{\ast}.
\]
Iterating this cancellation yields an ordered subsequence $(T^{\prime\prime})^{\ast} \mid S^{\ast}$ such that 
\[
\pi\bigl((T^{\prime\prime})^{\ast}\bigr)
= \pi\bigl((T')^{\ast}\bigr)
= 1.
\]
Thus $S^{\ast}$ contains a non-trivial product-one ordered subsequence in this case as well.

\smallskip
    
\no \textbf{Case (iii) :} If $S^{\ast}$ contains terms from both $H$ and $A$, write
\[
S^{\ast}(H)=c_1\bdot c_2\bdot\cdots\bdot c_u,\qquad
S^{\ast}(A)=d_1\bdot d_2\bdot\cdots\bdot d_v,
\]
with $u+v=\lvert S^{\ast}\rvert=\mathsf d(A)+2$.  Without loss of generality, assume that $c_i$ are distinct. Write $S^{\ast}$ as
\[
S^{\ast}=l^{\ast}_0\bdot c_1\bdot l^{\ast}_1\bdot c_2\bdot\cdots\bdot c_u\bdot l^{\ast}_u,
\]
where each $l^{\ast}_i$ is an  ordered subsequence of $S^{\ast}(A)$ of the form
\[
l^{\ast}_i=d_1^{e_{i1}}\bdot d_2^{e_{i2}}\bdot\cdots\bdot d_v^{e_{iv}},\qquad e_{ij}\in\{0,1\}.
\] 
Here, any $d_j$ with $e_{ij}=0$ is omitted from $l_i^{\ast}$, so $l_i^{\ast}$ consists precisely of those $d_j$ for which $e_{ij}=1$. If all $e_{ij}=0$ for a given $i$, then $l_i^{\ast}$ is the trivial sequence and contributes nothing at that position in $S^{\ast}$.

The subsequences $l_0^{\ast},\dots,l_u^{\ast}$ are pairwise disjoint, meaning that each $d_j$ appears in exactly one of these subsequences, and no $d_j$ is repeated across different $l_i^{\ast}$. Moreover, $\sum_{i=0}^u|l^{\ast}_i|=v$. Define the ordered sequence $L^{\ast}\in\mathcal F^{\ast}(A)$ by
\[
L^{\ast}
= l^{\ast}_0\bdot (c_1\pi(l^{\ast}_1)c_2)\bdot (c_2\pi(l^{\ast}_2)c_3)\bdot\cdots\bdot (c_{u-1}\pi(l^{\ast}_{u-1})c_u)\bdot l^{\ast}_u\bdot l^{\ast}_1\bdot\cdots\bdot l^{\ast}_{u-1},
\]
where each block \((c_i\pi(l^{\ast}_i)c_{i+1})\) is regarded as a single element of \(A\) (if \(l_i^{\ast}\) is trivial, then this block reduces to \((c_ic_{i+1})\)). Expanding the final concatenation shows
\[
|L^{\ast}| = \sum_{i=0}^u |l_i^{\ast}| + (u-1) = v + (u-1) = u+v-1 = \mathsf d(A)+1.
\]
Since $A$ is abelian, we have $\mathsf D_o(A)=\mathsf d(A)+1$, so any ordered sequence of length $\mathsf d(A)+1$ over $A$ contains a non-trivial product-one ordered subsequence. Hence there exists $(L')^{\ast}\mid L^{\ast}$ with $\pi((L')^{\ast})=1$. We now lift $(L')^{\ast}$ to an ordered subsequence $(L'')^{\ast}\mid S^{\ast}$ as follows. If $(c_ic_{i+1})\mid (L')^{\ast}$ then include the ordered pair $c_i\bdot c_{i+1}$ in $(L'')^{\ast}$.  If
\[
(c_i d_j d_{j+1}\cdots d_{j+k} c_{i+1})\bdot d_l\bdot d_{l+1}\bdot\cdots\bdot d_{l+s}\mid (L')^{\ast}
\]
with $J=[j,j+k]\supseteq I=[l,l+s]$, then include in $(L'')^{\ast}$ the ordered block
\[
c_i\bdot\biggl(\prod_{t\in J\setminus I} d_t\biggr)\bdot c_{i+1},
\]
where the product $\prod_{t\in J\setminus I} d_t$ is taken in the natural order; those $d_t$ appear in $S^{\ast}$ (in case \(J \setminus I\) is empty, the corresponding lifted block is simply \(c_i \bdot c_{i+1}\)). Because $cd=d^{-1}c$ for $c\in H,d\in A$ and $c^2=1$ for $c\in H$, the ordered product of the lifted sequence $(L'')^{\ast}$ equals the product of $(L')^{\ast}$.  Since $\pi((L')^{\ast})=1$, we obtain $\pi((L'')^{\ast})=1$, as required.

\no Moreover from Lemma \ref{beta}, we have $\mathsf{D}_o(G)=\beta (G).$ 
\qed

\subsection*{Proof of Theorem \ref{TheoremE(G)}}
Let $A$ be a finite abelian $p$-group with order $p^{\alpha},$ where $\alpha \in \mathbb{N}.$ Then, $A\cong C_{p^{\alpha_1}}\times C_{p^{\alpha_2}}\times \dotsc \times C_{p^{\alpha_r}},$ where $1 \le \alpha_1 \le \dotsc \le \alpha_r$, and $\alpha=\sum_{i=1}^r\alpha_i.$ Say, $C_{p^{\alpha_i}}= \langle g_i \rangle$ for each $i$, then $G=\langle g_1,\cdots,g_r,h \rangle$ where $h^2=1.$
Set $H:=G \setminus A$.
We want to show that $\mathsf{E}(G) \le \frac{3}{2}|G| $ for $G= A \rtimes_{-1} C_2.$ For $p=2$, the result follows from Lemma \ref{Han}. For $r=1$, the group $G$ is the dihedral group of order $2p^{\alpha_1}$, and the result follows from \cite[Theorem~8]{BAS2007}, so we henceforth assume that $p$ is an odd prime and $r \ge 2$.

\no Let $S$ be a sequence over $G$ of length $3p^{\alpha}$.  We consider
    \begin{align*}
    S(A)=(b_1 \bdot b_1) \bdot (b_2 \bdot b_2)\bdot \dotsc \bdot (b_t \bdot b_t)\bdot d_1\bdot \dotsc \bdot d_v,\\
S(H)=(a_1 \bdot a_1) \bdot (a_2 \bdot a_2)\bdot \dotsc \bdot (a_q \bdot a_q)\bdot c_1\bdot \dotsc \bdot c_u,
\end{align*} where $d_i$'s and $c_i$'s are distinct. Since $|S|=3p^{\alpha}=2q+2t+u+v$, and $u,v \le p^{\alpha}$, we have $2q+2t \ge p^{\alpha}$. We divide the proof into two cases:

\noindent {\bf Case 1:}  Let us assume that $q=0.$ Therefore, $|S(A)|=2t+v \ge 2p^{\alpha}$. From \cite{Olson1969I}, we have $\mathsf{E}(A)=p^{\alpha}+\sum_{i=1}^r p^{\alpha_i}-r \le 2p^{\alpha}-1$. So, there exists a subsequence $T_1 \mid S(A) \text{ such that } 1 \in \pi(T_1) \text{ and } |T_1|=p^{\alpha}.$

\no Now, consider the sequence $S^{\prime}=S(A) \bdot T_1^{[-1]}$ over $A$ of length $|S^{\prime}| \ge p^{\alpha}$. If $u \le 1$, then $|S(A)| \ge 3p^{\alpha}-1.$ Consequently, $|S^{\prime}| \ge 2p^{\alpha}-1 \ge \mathsf{E}(A).$ Thus, there exists a subsequence $T_2 \mid S^{\prime} \text{ such that } 1 \in \pi(T_2) \text{ and }   |T_2|=p^{\alpha}.$ Hence, we have a product-one subsequence $T_1 \bdot T_2 \mid S(A)$ of length $2p^{\alpha}.$

\no Now, assume $u\in [2, p^{\alpha}]$ and $1 \notin \Pi_{p^{\alpha}}(S^{\prime})$. Since $c_1c_i \neq c_1c_j$ for $i \neq j$, where $i,j\in [1, u]$, we have $$\big|\Pi_2(c_1 \bdot c_2\bdot \dotsc \bdot c_u)\big| \ge|\{c_1c_2,c_1c_3, \dotsc ,c_1c_u\}| = u-1.$$

\noindent \textbf{Subcase 1A:} For $\big|\Pi_2(c_1 \bdot c_2\bdot \dotsc \bdot c_u)\big| = u-1$.
For $i,j \in [1,u]$, we write $c_i=g_1^{k_{1i}}g_2^{k_{2i}}\dotsc g_r^{k_{ri}}h$ and $c_j=g_1^{k_{1j}}g_2^{k_{2j}}\dotsc g_r^{k_{rj}}h$. Since $h^2=1$, then 
$c_ic_j=g_1^{k_{1i}-k_{1j}}g_2^{k_{2i}-k_{2j}}\dotsc g_r^{k_{ri}-k_{rj}}.$
Now, consider the set $\{c_1c_2,c_1c_3,\dotsc,c_1c_u \}$ of $u-1$ elements. Rewriting these elements, we have:
 $$\{c_1c_2,c_1c_3,\dotsc,c_1c_u \}=\Big\{ \prod_{i=1}^rg_i^{k_{i1}-k_{i2}},\cdots, \prod_{i=1}^rg_i^{k_{i1}-k_{iu}}\Big\}.$$

\no Similarly, the set $\{c_2c_1,c_2c_3,\dotsc,c_2c_u \}=\Big\{ \prod_{i=1}^rg_i^{k_{i2}-k_{i1}},\cdots, \prod_{i=1}^rg_i^{k_{i2}-k_{iu}}
\Big\}$ contains $u-1$ distinct elements. Hence, both sets are equal. Therefore,

$$uk_{m1} \equiv uk_{m2}\ (\textrm{mod}\ p^{\alpha_m})\quad\text{for all}\; m \in [1, r].$$
More generally, $uk_{ms} \equiv uk_{ml}\ (\textrm{mod}\ p^{\alpha_m}) \;\text{for all}\; m\in [1,r],\;\text{and } s,l\in [1,u].$

\no Note that since $ u \in [2,p^{\alpha}]$, it follows that $\gcd(p,u) \in \{1,p\}$.

\no {\bf Subcase 1A (i):} Assume $\gcd(p,u)=p$. Then $u \in \{p,p^2,\dotsc,p^{\alpha}\}$. 

\no If $u=p^j$ for some $ j \in [1,\alpha-1]$, then $|S^{\prime}|= 2p^{\alpha}-p^j.$ 
We first show that $|S'| \ge 1 + \mathsf{d}(A \times C_{p^{\alpha}})$. This reduces to verifying that $p^{\alpha} - p^{\alpha-1} \ge r p^{\alpha_r} - r$, or equivalently that $p^{\alpha} - p^{\alpha-1} - r p^{\alpha_r} \ge -r$. Since $p^{\alpha} - p^{\alpha-1} = p^{\alpha_r}(p^{\alpha-\alpha_r} - p^{\alpha-\alpha_r-1})$, it is enough to check that $p^{\alpha-\alpha_r} - p^{\alpha-\alpha_r-1} - r \ge 0$. Noting that $p^{\alpha-\alpha_r-1}(p-1) \ge p^{r-2}(p-1) \ge 2p^{r-2} \ge 2\cdot 3^{r-2} \ge r$ for all $p \ge 3$ and $r \ge 2$, the inequality follows, and hence $|S'| \ge 1 + \mathsf{d}(A \times C_{p^{\alpha}})$.

As $\mathsf{d}(A \times C_{p^{\alpha}})\le 3p^{\alpha}-2$, by Lemma \ref{n,2n}, we can say that there exists a subsequence $T_2 \mid S^{\prime} \text{ such that } 1\in \pi(T_2) \text{ and } |T_2|= p^{\alpha}.$ Therefore, $T_1 \bdot T_2$ is a product-one subsequence of $S$ of length $2p^{\alpha}$.

\no For $u=p^{\alpha}$, we have $|S^{\prime}|=p^{\alpha}$. Then by Lemma \ref{pi_n-2}, we have 
$$\big|\Pi_{p^{\alpha}-2}(S^{\prime})\big|=\big|\Pi_{2}(S^{\prime})\big| \ge 1.$$
If $\big|\Pi_{p^{\alpha}-2}(S^{\prime})\big|=\big|\Pi_{2}(S^{\prime})\big| > 1$, we are done. However, if $\big|\Pi_{p^{\alpha}-2}(S^{\prime})\big|=\big|\Pi_{2}(S^{\prime})\big|= 1$, then all $p^{\alpha}$ elements of $S^{\prime}$ are the same, which contradicts the condition $1 \notin \Pi_{p^{\alpha}}(S^{\prime})$.

\no {\bf Subcase 1A (ii):} $\gcd(p,u)=1$.

\no In this case, we have $k_{ms} \equiv k_{ml}\ (\textrm{mod}\ p^{\alpha_m})\;\text{for all}\; m \in [1, r],\; s,l \in [1, u]$. This implies $c_sc_l=1 \;\text{for all}\;  s\neq l$ , a contradiction to the subsequence $S(H)$.

\noindent \textbf{Subcase 1B:} For $\big|\Pi_2(c_1 \bdot c_2\bdot \dotsc \bdot c_u)\big| \ge u$.

\no Since $S^{\prime}$ is a sequence over $A$ with $|S^{\prime}|=2p^{\alpha}-u.$ Using Lemma \ref{pi_n-2}, it follows that
$$\big|\Pi_{p^{\alpha}-2}(S^{\prime})\big| = \big|\Pi_{p^{\alpha}-u+2}(S^{\prime})\big| \ge p^{\alpha}-u+1.$$ Hence, by Lemma \ref{Nathanson}, we obtain a product-one subsequence of $S$ of length $2p^{\alpha}$.

\no {\bf Case 2:} Let us assume that $q \geq 1.$ If $(u,v)=(0,0)$, then $|S|=2q+2t=3p^{\alpha}$, a contradiction. 

\no So, one of $u$ and $v$ is non-zero. Assume $v=0$ and $u$ is non-zero. Since $|S|=3p^{\alpha}=2q+2t+u$, implies $u$ is an odd integer. We can choose integers $q^{\prime} \in [1,q]$ and $t^{\prime} \in [0,t]$ such that $2q^{\prime}+2t^{\prime}=2p^{\alpha}$. Then consider the subsequence $T$ of $S$ as follows:
$$T=(b_1 \bdot b_2 \bdot \dotsc \bdot b_{t^{\prime}})\bdot a_1 \bdot (b_1 \bdot b_2 \bdot \dotsc \bdot b_{t^{\prime}})\bdot a_1 \bdot (a_2\bdot a_2) \bdot (a_3\bdot a_3)\bdot (a_{q^{\prime}}\bdot a_{q^{\prime}}).$$ Clearly, $|T|=2p^{\alpha}$ and $ \pi(T^{\ast}) =1. $ Similar conclusion can be drawn if $u=0$ and $v$ is non-zero.

\no Now, assume that both $u$ and $v$ are non-zero. Since $u+v+2q+2t=3p^{\alpha}$, both $u$ and $v$ must have different parities. W.l.o.g., we assume that $u=2l$ and $v=2k+1$ for some $l,k \in \mathbb{Z}$.
This implies $2l+2k=3p^{\alpha}-2q-2t-1=u+v-1.$ Consider the following two sequences $C$ and $D$ over $A$ such that  $|C|=l$ and $|D|=k$.
 \begin{align*}
     C&=(c_1  c_2) \bdot (c_3  c_4)\bdot \dotsc \bdot (c_{2l-1}  c_{2l}),\\
     D&=(d_1  d_2) \bdot (d_3  d_4)\bdot \dotsc \bdot (d_{2k-1}  d_{2k}).
 \end{align*}
 
\no Since $2^{\lceil \frac{p^{\alpha}}{2} \rceil} > p^{\alpha}$, by applying Lemma \ref{disjointsequences} to the sequence $C\bdot D$, we have two disjoint subsequences $U,V \mid C\bdot D$ such that $\pi(U)=\pi(V)$ and 
 $$l+k-|U|-|V| \le \bigg\lceil \frac{p^{\alpha}}{2} \bigg\rceil -1.$$
\no Note that $l+k+1-\bigg\lceil \frac{p^{\alpha}}{2} \bigg\rceil \le |U|+|V| \le |C \bdot D|=l+k.$ Thus
\begin{align*}
    3p^{\alpha}-2q-2t+1-2 \bigg\lceil \frac{p^{\alpha}}{2} \bigg\rceil 
    &\le 2|U|+ 2|V|\\
    &\le 2l+2k\\
    &= u+v-1\\
    &\le 2p^{\alpha}-1.
\end{align*}
Since $1-2\lceil \frac{p^{\alpha}}{2} \rceil=-p^{\alpha}$, we obtain
$2|U|+ 2|V|+2q+2t \ge 2p^{\alpha}.$ So, we choose $q^{\prime} \in [1, q]$ and $ t^{\prime} \in [0,t]$ such that $2|U|+ 2 |V|+2q^{\prime}+2t^{\prime}=2p^{\alpha}$. Since $U \mid (C\bdot D)$, so there is a corresponding subsequence $U^{\prime} \mid S$, such that if $(c_ic_{i+1}) \mid U$, then $ c_i \bdot c_{i+1} \mid U^{\prime}$. Similarly, if $(d_id_{i+1}) \mid U$, then $ d_i \bdot d_{i+1} \mid U^{\prime}$. Likewise, for the subsequence $V \mid (C\bdot D)$, there is a corresponding subsequence $V^{\prime} \mid S$. 
Thus, $|U^{\prime}|=2|U|$ and $|V^{\prime}|=2|V|$, resulting in $|U^{\prime}|+|V^{\prime}|+2q^{\prime}+2t^{\prime}=2p^{\alpha}$. Now, let us consider the sequence
$$T=(b_1 \bdot b_2 \bdot \dotsc \bdot b_{t^{\prime}}) \bdot U^{\prime } \bdot a_1 \bdot (b_1 \bdot b_2 \bdot \dotsc \bdot b_{t^{\prime}}) \bdot V^{\prime} \bdot a_1 \bdot (a_2\bdot a_2) \bdot (a_3\bdot a_3)\bdot (a_{q^{\prime}}\bdot a_{q^{\prime}})$$ of length $2p^{\alpha}$. Clearly, $\pi(T^{\ast})=1.$ This completes the proof.
 \qed

\section{Results related to $\mathsf{L}(G)$}\label{SectionL(G)}

\no In this section, we will focus on computing the Loewy length of the group $G_i$ for $i\in [1,4]$.

\begin{lem}\label{LemmaG_1} For the group $G=G_1$, we have
$$\mathsf{L}(G)= 2^\alpha +2^\beta +2^{\gamma+1}-3.$$
    
\end{lem}
\begin{proof} According to Lemma \ref{JenningsTheorem} and Proposition \ref{Ind}, the quotients  $G^{2^k}/G^{2^{k+1}}$ are elementary abelian $2$-groups, with a rank $r(G^{2^k})$ for all $k \in [1,\alpha-1]$. We claim that $r(G^{2^k}) = 3$ for all $k \in [1,\gamma]$.  
The set of generators of $G^{2^k}$ is $\{a^{2^k}, b^{2^k}, c^{2^{k-1}}\}$, and its proper subsets are
\[
\{a^{2^k}, b^{2^k}\},\quad 
\{a^{2^k}, c^{2^{k-1}}\},\quad
\{b^{2^k}, c^{2^{k-1}}\},\quad
\{a^{2^k}\},\quad
\{b^{2^k}\},\quad
\{c^{2^{k-1}}\}.
\]
Consider $A = \{a^{2^k}, b^{2^k}\}$. We claim that $\langle A \rangle \neq G^{2^k}$.  
If, on the contrary, $\langle A \rangle = G^{2^k}$, then by Proposition~\ref{power} we would have  
$c^{2^{k-1}} \in \langle A \rangle$, which cannot occur because  
$\langle a, c \rangle \cap \langle b \rangle = 1$.  
Thus $\langle A \rangle \neq G^{2^k}$, and an identical argument applies to every proper subset $A$.  
Therefore, none of them generate $G^{2^k}$, and we conclude that $r(G^{2^k}) = 3$.

\no Similarly, using the facts that $\langle a,c \rangle \cap \langle b \rangle = 1$ and $\langle a \rangle \cap \langle c \rangle = 1$, the rank computations for other values of $k$ follow analogously, yielding
 \begin{equation*}
     r(G^{2^k})=
     \begin{cases}
         3 & \text{ if } k \in [1, \gamma], \\ 
         2 & \text{ if } k \in [\gamma+1, \beta-1], \\
         1 & \text{ if } k \in [\beta, \alpha-1].
     \end{cases}
 \end{equation*}
 
\noindent Thus, whenever $\alpha\geq\beta>\gamma$, we have
$$|G^{2^k}:G^{2^{k+1}}|=
\begin{cases}
2^3 &\mbox { if } k \in [1, \gamma], \\
2^2 &\mbox { if } k \in [\gamma +1, \beta -1], \\
2 &\mbox { if } k\in [\beta, \alpha -1],
\end{cases}
$$
\noindent and, for $\alpha>\beta=\gamma$, we obtain
$$|G^{2^k}:G^{2^{k+1}}|=
\begin{cases}
2^3 &\mbox { if } k \in [1, \gamma -1], \\
2^2 &\mbox { if } k=\gamma, \\
2 &\mbox { if } k \in [\gamma +1, \alpha -1].
\end{cases}
$$
\noindent On the other hand, if $\alpha=\beta=\gamma$, then
$$|G^{2^k}:G^{2^{k+1}}|=
\begin{cases}
2^3 &\mbox { if } k \in [1, \gamma -1], \\
2 &\mbox { if } k=\gamma.
\end{cases}
$$
\noindent Since $|M_{2^s}:M_{2^{s}+1}|=|G^{2^s}:G^{2^{s+1}}|$ for any $s\geq 1$, applying Lemma \ref{JenningsTheorem} with
\begin{equation*}
    d=
    \begin{cases}
        2^{\gamma} & \text{ if } \alpha=\beta=\gamma,\\
        2^{\alpha-1} & \text{ otherwise }
    \end{cases}
\end{equation*} implies $\mathsf{L}(G)= 2^\alpha +2^\beta +2^{\gamma+1}-3$. 
\end{proof}

\begin{lem}\label{LemmaG_2} For the group $G=G_2$, we have
$$\mathsf{L}(G)=2^\alpha +2^\beta-1.$$
\end{lem}
\begin{proof}
Let us consider the case $\alpha\geq\beta.$ Using Proposition \ref{Ind}, we have 
$$M_k=
\begin{cases}
G &\mbox{if } k=1,\\
G^{2^i} &\mbox{if }  k \in [2^{i-1}+1, 2^i] \text{ for all } i \in [1, \alpha-1],\\
1 &\mbox{if }  k \in [2^{\alpha-1}+1, 2^{\alpha}].
\end{cases}
$$
Since $\langle a \rangle\cap \langle b\rangle=1$, the quotient groups $G^{2^k}/G^{2^{k+1}}$ are elementary abelian $2$-groups with
\begin{equation*}
  r(G^{2^k})=
  \begin{cases}
    2 & \text{ if }  k \in [1, \beta-1],\\
    1 & \text{ if }  k\in [\beta, \alpha-1].
  \end{cases}
\end{equation*}
Therefore
$$|G^{2^k}:G^{2^{k+1}}|=
\begin{cases}
2^2 &\mbox { if } k \in [1, \beta -1], \\
2 &\mbox { if } k \in [\beta, \alpha -1].
\end{cases}
$$
\noindent Now, setting $d=2^{\alpha-1}$ in Lemma \ref{JenningsTheorem}, we obtain $\mathsf{L}(G)= 2^\alpha +2^\beta-1$.
The case $\beta\geq \alpha$ is similar.  
\end{proof}

\begin{lem}\label{LemmaG_3} For the group $G=G_3$, we have
$$\mathsf{L}(G)= 2^\alpha +2^\beta+2^{\sigma+1}-3.$$ 
\end{lem}

\begin{proof}
Assume that $\alpha\geq\beta.$ The relation $[a,b]=a^{2^{\alpha-\gamma}} c$\; implies that $G^{2^k}= \langle a^{2^k}, b^{2^k}, c^{2^{k-1}} \rangle$ for all $k\geq 1.$ Similar to Lemma \ref{LemmaG_1}, we have
$$|G^{2^k}:G^{2^{k+1}}|=
\begin{cases}
2^3 &\mbox { if } k \in [1, \sigma], \\
2^2 &\mbox { if } k \in [\sigma +1, \beta -1], \\
2 &\mbox { if } k \in [\beta, \alpha -1].
\end{cases}
$$
Using Proposition \ref{Ind}, and setting $d=2^{\alpha-1}$ in Lemma \ref{JenningsTheorem}, we find $\mathsf{L}(G)= 2^\alpha +2^\beta+2^{\sigma+1}-3.$ Similarly, the case $\beta<\alpha$ follows.
\end{proof}

\begin{lem}\label{LemmaG_4} For the group $G=G_4$, we have
$$\mathsf{L}(G)= 2^{\gamma+2}-3.$$
\end{lem}

\begin{proof}
 Using Proposition \ref{power} and the relation  $[a,b]=a^{2} c$, we have $G^{2^k}=\langle a^{2^k}, b^{2^k}, c^{2^{k-1}} \rangle$ for all $k \in [1,\gamma].$ Set $x=a^{2^k}, y=b^{2^k},\;\text{and}\; z=c^{2^{k-1}}.$ Then we have the following relations:
\[[x,y]=[y,z]=(xz)^{2^{k+1}}, \quad [x,z]=1, \quad x^{2^{\gamma-k}}=y^{2^{\gamma-k}}.\]
\noindent Consider the quotient group $\overline{G^{2^k}}=\frac{G^{2^k}}{\langle x,z\rangle}.$  
 Observe that  $\overline{G^{2^k}}\cong \langle \overline{y} \rangle$. We will first show that $o(\overline{y})=2^{\gamma-k}$. Suppose that $\overline{y}^{2^l}=\overline{1}$ for some $l<\gamma-k.$ This implies $y^{2^l}\in \langle x,z\rangle$ and hence $y^{2^l}\in \langle a,c\rangle.$ As $y=b^{2^k}$, we get $b^{2^{k+l}} \in \langle a,c\rangle.$ This implies $\overline{b}^{2^{k+l}}=\overline{1}$ in $\frac{G}{\langle a,c\rangle}.$ But  $|G|=2^{3 \gamma},$ implies that $|\frac{G}{\langle a,c\rangle}|=o(\overline{b})=2^\gamma.$ So, the claim follows. Therefore, 
  $|G^{2^k}|=2^{3(\gamma-k)+1}$ for all $k \in [1, \gamma].$
 
\no Hence
$$|G^{2^k}:G^{2^{k+1}}|=
\begin{cases}
2^3 &\mbox { if } k\in [1, \gamma -1], \\
2 &\mbox { if } k=\gamma.
\end{cases}
$$
\noindent Now, setting $d=2^{\gamma}$ in Lemma \ref{JenningsTheorem}, we obtain $\mathsf{L}(G)= 2^{\gamma+2}-3.$ 
\end{proof}

\section{Proof of Theorem \ref{TheoremOrderedLoewy} and Theorem \ref{LemmaIndexp}}

\subsection*{Proof of Theorem \ref{TheoremOrderedLoewy}}
\begin{enumerate}[label={(\arabic*)}]

\item Let $G=G_1.$ By Lemma \ref{Dimitovtheorem} and Lemma \ref{LemmaG_1}, we have $\mathsf{D}_o(G) \le 2^\alpha +2^\beta +2^{\gamma+1}-3$. To prove the reverse inequality, we consider 
the ordered sequence $$S^{\ast}=(a^{-1})^{[2^{\alpha}-1]}\bdot  (b^{-1})^{[2^{\beta}-1]}\bdot  (a[a,b])^{[2^{\gamma}-1]}\bdot  (b[a,b])^{[2^{\gamma}-1]}$$ over $G$ of length $|S^{\ast}| = 2^{\alpha}+2^{\beta}+ 2^{\gamma+1}-4$. If possible, let us assume that $S^{\ast}$ has a non-trivial product-one ordered subsequence. Then
\begin{equation}\label{LBEQ1}
       (a^{-1})^{x}(b^{-1})^{y}(a[a,b])^z(b[a,b])^w=1 
    \end{equation} 
with $x \in [0,(2^{\alpha}-1)], \; y \in [0,(2^{\beta}-1)], \; z,w\in [0, (2^{\gamma}-1)],$ with not all zero.
Equation $(\ref{LBEQ1})$ results in $a^{-x+z}b^{-y+w}[a,b]^{z+w+yz}=1$. So, we have the following system of equations
 \begin{align*}
  -x+z &\equiv 0\ (\textrm{mod}\ 2^{\alpha})\\ 
  -y+w & \equiv 0\ (\textrm{mod}\ 2^{\beta}) \\ 
z+w+yz&\equiv 0\ (\textrm{mod}\ 2^{\gamma}).
\end{align*}
If $\gamma=1$, then the trivial solution is the only solution to the above system of equations, providing a contradiction. Thus, $ \mathsf{D}_o(G) \ge 2^\alpha +2^\beta +2^{\gamma+1}-3.$

\item Let $G=G_2.$ From Lemma \ref{Dimitovtheorem} and Lemma \ref{LemmaG_2}, we have $\mathsf{D}_o(G) \le 2^\alpha +2^\beta-1$. To prove the reverse inequality, we consider the ordered sequence $S^{\ast}= a^{[2^{\alpha}-1]}\bdot b^{[2^{\beta}-1]}$ over $G$ of length $|S^{\ast}|=2^{\alpha}+2^{\beta}-2$. Since $\langle a \rangle \cap \langle b \rangle =1$, the ordered sequence $S^{\ast}$ has no non-trivial product-one ordered subsequence. Hence, $\mathsf{D}_o(G) \ge 2^{\alpha}+2^{\beta}-1 .$

    \item Let $G=G_3.$ Lemma \ref{Dimitovtheorem} and Lemma \ref{LemmaG_3} implies that $\mathsf{D}_o(G) \le 2^\alpha +2^\beta+2^{\sigma+1}-3$.
 For the reverse inequality, consider the ordered sequence $$S^{\ast}= (a^{-1})^{[2^{\alpha}-1]} \bdot  (b^{-1})^{[2^{\beta}-1]} \bdot  (a[a,b])^{[2^{\sigma}-1]} \bdot (b[a,b])^{[2^{\sigma}-1]}$$ over $G$ of length $|S^{\ast}|=2^{\alpha}+2^{\beta}+ 2^{\sigma+1}-4$. Suppose $S^{\ast}$ has a non-trivial product-one ordered subsequence. In this case, we have
    \begin{equation}\label{LBEQ3}
       (a^{-1})^{x}(b^{-1})^{y}(a[a,b])^z(b[a,b])^w=1 
    \end{equation} for $x \in [0, (2^{\alpha}-1)],\; y\in [0, (2^{\beta}-1)], \; z,w \in [0, (2^{\sigma}-1)],$ with not all zero. Equation $(\ref{LBEQ3})$ leads to $a^{-x+z}b^{-y+w}[a,b]^{z+w+yz}=1.$

\noindent Let us consider $\overline{G}:=\frac{G}{\langle a,c\rangle}$, then $(\overline{b})^{-y+w}=\overline{1}$. Since $\overline{G}\cong C_{2^{\beta}}$, we have $-y+w \equiv 0\ (\textrm{mod}\ 2^{\beta}).$ Thus, the above equation simplifies to $a^{-x+z}[a,b]^{z+w+yz}=1$. Using the relation $[a,b]=a^{2^{\alpha-\gamma}} c$, we obtain following system of equations
 \begin{align*}
  -x+z+2^{\alpha-\gamma}(z+w+yz) &\equiv 0\ (\textrm{mod}\ 2^{\alpha})\\ 
  -y+w & \equiv 0\ (\textrm{mod}\ 2^{\beta}) \\ 
z+w+yz&\equiv 0\ (\textrm{mod}\ 2^{\sigma}).
\end{align*}
If $\sigma=1$, the only solution to the given system of equations is the trivial solution, which leads to a contradiction. This yields $ \mathsf{D}_o(G) \ge 2^\alpha +2^\beta+2^{\sigma+1}-3 .$
    \item Let $G=G_4.$ Using Lemma \ref{Dimitovtheorem} and Lemma \ref{LemmaG_4}, we have $\mathsf{D}_o(G) \le 2^{\gamma+2}-3.$ For the reverse inequality, we consider the ordered sequence $$S^{\ast}= (a^{-1})^{[2^{\gamma+1}-1]}\bdot (b^{-1})^{[2^{\gamma}-1]} \bdot (a[a,b])^{[2^{\gamma-1}-1]}\bdot (b[a,b])^{[2^{\gamma-1}-1]}$$ over $G$ of length $|S^{\ast}|=2^{\gamma+2}-4.$ Suppose $S^{\ast}$ has a non-trivial product-one ordered subsequence. Therefore,
\begin{equation}\label{LBEQ4}
     (a^{-1})^{x}(b^{-1})^{y}(a[a,b])^z(b[a,b])^w=1
     \end{equation} for $x \in [0, (2^{\gamma+1}-1)], \; y \in [0, (2^{\gamma}-1)], \; z,w\in [0, (2^{\gamma-1}-1)],$ with not all zero. Equation $(\ref{LBEQ4})$ leads to $a^{-x+z}b^{-y+w}[a,b]^{z+w+yz}=1.$

Let us consider $\overline{G}:=\frac{G}{\langle a,c\rangle}$, then we have $-y+w \equiv 0\ (\textrm{mod}\ 2^{\gamma})$. Thus, the above equation reduces to $a^{-x+z}[a,b]^{z+w+yz}=1$, using the  given relation $[a,b]=a^2 c$, we have following system of equations
 \begin{align*}
  -x+z+2(z+w+yz) &\equiv 0\ (\textrm{mod}\ 2^{\gamma  +1})\\ 
  -y+w & \equiv 0\ (\textrm{mod}\ 2^{\gamma}) \\ 
z+w+yz&\equiv 0\ (\textrm{mod}\ 2^{\gamma-1}).
\end{align*}
If $\gamma \in \{1,2\}$, then the system admits only the trivial solution, which gives a contradiction. Therefore, $\mathsf{D}_o(G) \le 2^{\gamma+2}-3.$ \qed

\end{enumerate}

\no \textbf{Proof of Corollary \ref{d}}

\no The sequence $S=a^{[2^{\alpha}-1]}\bdot b^{[2^{\beta}-1]}$ over $G$ is a product-one free sequence, by Lemma \ref{d(G)D_o(G)}, we conclude $2^\alpha +2^\beta-2 = \mathsf{d}(G)$. \qed
 
\subsection*{Proof of Theorem \ref{LemmaIndexp}}
Assume that $G$ has a cyclic subgroup of index $p$. By Lemma \ref{EqL}, we have $\mathsf{L}(G)=\frac{|G|}{p}+p-1$. Using Lemma \ref{dd}, Lemma \ref{d(G)D_o(G)}, and Lemma \ref{Dimitovtheorem}, we conclude 
$$\mathsf{d}(G)+1=\mathsf{D}_o(G)= \mathsf{L}(G)=\frac{|G|}{p}+p-1.$$
Conversely, since $\mathsf{L}(G)=\frac{|G|}{p}+p-1,$ so we have the result using Lemma \ref{EqL}. \qed

\section{Concluding Remarks}
In this paper, we concentrated on several zero-sum invariants and explored their interrelationships. In Theorem \ref{TheoremE(G)}, we identify a class of non-abelian groups of the form $G = A \rtimes_{-1} C_2$, where $A$ is a finite abelian $p$-group, for which Conjecture \ref{GAOLI} holds. This naturally leads to the question of whether Conjecture \ref{GAOLI} continues to hold when the class is extended to include all finite groups of the form $G = A \rtimes C_q$, where $q$ is a prime and $A$ is a finite abelian group. 
\section{Acknowledgements}
We thank the anonymous referee for the careful reading of the manuscript and for the helpful comments that improved its overall exposition. The second author acknowledges support from the Alexander von Humboldt Foundation through a Humboldt Research Fellowship for Postdocs (Humboldt ID number 1243434), and the third author from the ANRF Core Research Grant (File No. CRG/2023/002698).


\bigskip







\begin{thebibliography}{10}

\bibitem{ABR2023}
D.~V. Avelar, F.~E. Brochero~Mart\'{\i}nez, and S.~Ribas.
\newblock On the direct and inverse zero-sum problems over {$C_n\rtimes_sC_2$}.
\newblock {\em J. Combin. Theory Ser. A}, 197:Paper No. 105751, 18, 2023.

\bibitem{BK}
Michael~R. Bacon and Luise-Charlotte Kappe.
\newblock The nonabelian tensor square of a {$2$}-generator {$p$}-group of class {$2$}.
\newblock {\em Arch. Math. (Basel)}, 61(6):508--516, 1993.

\bibitem{BAS2007}
Jared Bass.
\newblock Improving the {E}rd\"{o}s-{G}inzburg-{Z}iv theorem for some non-abelian groups.
\newblock {\em J. Number Theory}, 126(2):217--236, 2007.

\bibitem{chapman}
Scott~T. Chapman, Michael Freeze, and William~W. Smith.
\newblock Minimal zero-sequences and the strong {D}avenport constant.
\newblock {\em Discrete Math.}, 203(1-3):271--277, 1999.

\bibitem{CD2014}
K\'{a}lm\'{a}n Cziszter and M\'{a}ty\'{a}s Domokos.
\newblock The {N}oether number for the groups with a cyclic subgroup of index two.
\newblock {\em J. Algebra}, 399:546--560, 2014.

\bibitem{Davenport}
H.~Davenport.
\newblock Midwestern conference on group theory and number theory.
\newblock {\em Ohio State University}, 1966.

\bibitem{Dimitrov}
Vesselin Dimitrov.
\newblock On the strong {D}avenport constant of nonabelian finite $p$-groups.
\newblock {\em Math. Balkanica (N.S.)}, 18(1-2):131--140, 2004.

\bibitem{EGZ1961}
P.~Erd\H{o}s, A.~Ginzburg, and A.~Ziv.
\newblock Theorem in the additive number theory.
\newblock {\em Bull. Res. Council Israel Sect. F}, 10F(1):41--43, 1961.

\bibitem{GAO1996B}
Weidong Gao.
\newblock An improvement of the {E}rd\"{o}s-{G}inzburg-{Z}iv theorem.
\newblock {\em Acta Math. Sinica (Chinese Ser.)}, 39(4):514--523, 1996.

\bibitem{GLI2010}
Weidong Gao and Yuanlin Li.
\newblock The {E}rd\"{o}s-{G}inzburg-{Z}iv theorem for finite solvable groups.
\newblock {\em J. Pure Appl. Algebra}, 214(6):898--909, 2010.

\bibitem{GaoLiQu}
Weidong Gao, Yuanlin Li, and Yongke Qu.
\newblock On the invariant $\mathsf{E}({G})$ for groups of odd order.
\newblock {\em Acta Arith.}, 201(3):255--267, 2021.

\bibitem{GaoLu2008}
Weidong Gao and Zaiping Lu.
\newblock The {E}rd\"{o}s-{G}inzburg-{Z}iv theorem for dihedral groups.
\newblock {\em J. Pure Appl. Algebra}, 212(2):311--319, 2008.

\bibitem{GerolHalt}
Alfred Geroldinger and Franz Halter-Koch.
\newblock {\em Non-unique factorizations}, volume 278 of {\em Pure and Applied Mathematics (Boca Raton)}.
\newblock Chapman \& Hall/CRC, Boca Raton, FL, 2006.
\newblock Algebraic, combinatorial and analytic theory.

\bibitem{NRE}
Naveen~K. Godara, Renu Joshi, and Eshita Mazumdar.
\newblock An algebraic approach towards a conjecture on the {D}avenport constant.
\newblock {\em \emph{arXiv preprint:} https://arxiv.org/abs/2407.01148}, 2024.

\bibitem{NKG}
Naveen~K. Godara and Siddhartha Sarkar.
\newblock A note on a conjecture of {G}ao and {Z}huang for groups of order 27.
\newblock {\em J. Algebra Appl.}, 24(11):Paper No. 2550268, 21, 2025.

\bibitem{HanD}
Dongchun Han.
\newblock The {E}rd\"{o}s-{G}inzburg-{Z}iv theorem for finite nilpotent groups.
\newblock {\em Arch. Math. (Basel)}, 104(4):325--332, 2015.

\bibitem{DHA2019}
Dongchun Han and Hanbin Zhang.
\newblock {E}rd\"{o}s-{G}inzburg-{Z}iv theorem and {N}oether number for {$C_m\ltimes_\varphi C_{mn}$}.
\newblock {\em J. Number Theory}, 198:159--175, 2019.

\bibitem{J}
S.~A. Jennings.
\newblock The structure of the group ring of a {$p$}-group over a modular field.
\newblock {\em Trans. Amer. Math. Soc.}, 50:175--185, 1941.

\bibitem{Koshitani1977}
Shigeo Koshitani.
\newblock On the nilpotency indices of the radicals of group algebras of {$p$}-groups which have cyclic subgroups of index {$p$}.
\newblock {\em Tsukuba J. Math.}, 1:137--148, 1977.

\bibitem{Nathanson1996}
Melvyn~B. Nathanson.
\newblock {\em Additive number theory}, volume 165 of {\em Graduate Texts in Mathematics}.
\newblock Springer-Verlag, New York, 1996.
\newblock Inverse problems and the geometry of sumsets.

\bibitem{Olson1969I}
John~E. Olson.
\newblock A combinatorial problem on finite {A}belian groups. {I}.
\newblock {\em J. Number Theory}, 1:8--10, 1969.

\bibitem{OLS1976}
John~E. Olson.
\newblock On a combinatorial problem of {E}rd\"{o}s, {G}inzburg, and {Z}iv.
\newblock {\em J. Number Theory}, 8(1):52--57, 1976.

\bibitem{OW}
John~E. Olson and Edward~T. White.
\newblock Sums from a sequence of group elements.
\newblock In {\em Number theory and algebra}, pages 215--222. Academic Press, New York-London, 1977.

\bibitem{YYU2023}
Yongke Qu and Yuanlin Li.
\newblock On a conjecture of {Z}huang and {G}ao.
\newblock {\em Colloq. Math.}, 171(1):113--126, 2023.

\bibitem{QuLiTeeuwsen2022}
Yongke Qu, Yuanlin Li, and Daniel Teeuwsen.
\newblock On a conjecture of the small {D}avenport constant for finite groups.
\newblock {\em J. Combin. Theory Ser. A}, 189:Paper No. 105617, 14, 2022.

\bibitem{Schmid}
Barbara~J. Schmid.
\newblock Finite groups and invariant theory.
\newblock In {\em Topics in invariant theory ({P}aris, 1989/1990)}, volume 1478 of {\em Lecture Notes in Math.}, pages 35--66. Springer, Berlin, 1991.

\bibitem{YUS1988}
Thomas Yuster.
\newblock Bounds for counterexamples to addition theorems in solvable groups.
\newblock {\em Arch. Math. (Basel)}, 51(3):223--231, 1988.

\end{thebibliography}
\end{document}